\documentclass[final,12pt]{amsart}

 \title[Endomorphism algebras in negative CY categories]{Endomorphism algebras for a class of negative Calabi-Yau categories}

\author{Raquel Coelho Sim\~oes}
\email{rcoelhosimoes@campus.ul.pt}

\author{Mark James Parsons}
\email{markjamesparsons@googlemail.com}

\usepackage{amssymb}
\usepackage{latexsym}
\usepackage{amsmath}
\usepackage{mathrsfs}
\usepackage[all]{xy}
\usepackage{enumerate}	
\usepackage{enumitem}
\usepackage{paralist}
\usepackage{graphicx}

\usepackage{verbatim}

\usepackage{hyperref}
\newcommand{\harxiv}[1]{ \href{http://arxiv.org/abs/#1}{\texttt{arXiv:#1}}}
\newcommand{\hyref}[2]{ \hyperref[#2]{#1~\ref*{#2}} }

\usepackage[notref, notcite]{showkeys} 

\theoremstyle{plain}
\newtheorem{theorem}{Theorem}[section]

\newtheorem{lemma}[theorem]{Lemma}
\newtheorem{corollary}[theorem]{Corollary}
\newtheorem{proposition}[theorem]{Proposition}

\theoremstyle{definition}
\newtheorem{remark}[theorem]{Remark}
\newtheorem{example}[theorem]{Example}
\newtheorem*{definition}{Definition} % no numbering for definition.
\newtheorem*{definitions}{Definitions} % no numbering for definitions.

\newtheorem*{notation}{Notation}

\newcommand{\Case}[1]{\medskip\noindent\emph{Case #1:}}

\DeclareMathAlphabet{\mathpzc}{OT1}{pzc}{m}{it}

\newcommand{\cF}{\mathscr{F}}

\newcommand{\sB}{\mathsf{B}}

\newcommand{\sD}{\mathsf{D}}

\newcommand{\sG}{\mathsf{G}}

\newcommand{\sI}{\mathsf{I}}

\newcommand{\sR}{\mathsf{R}}

\newcommand{\sT}{\mathsf{T}}

\newcommand{\Db}{\mathsf{D}^b}

\newcommand{\bN}{\mathbb{N}}

\newcommand{\mcC}{\mathcal{C}}

\newcommand{\mcP}{\mathcal{P}}

\newcommand{\mcR}{\mathcal{R}}

\newcommand{\mcT}{\mathcal{T}}

\renewcommand{\geq}{\geqslant}
\renewcommand{\leq}{\leqs}
\renewcommand{\phi}{\varphi}
\renewcommand{\epsilon}{\varepsilon}

\DeclareMathOperator{\Ext}{\mathsf{Ext}}

\DeclareMathOperator{\End}{{\mathsf{End}}}
\newcommand{\ind}[1]{\mathsf{ind}(#1)}

\newcommand{\kk}{{\mathbf{k}}}

\newcommand{\leqs}{\leqslant}

\newcommand{\arc}[1]{\mathtt{#1}}

\entrymodifiers={+!!<0pt,\fontdimen22\textfont2>}

\newcommand{\arup}{\ar@/^/[u]}
\newcommand{\ardn}{\ar@/^/[d]}

\newcommand{\coloneqq}{\mathrel{\mathop:}=}

\usepackage{tikz,tkz-euclide}
\usepackage{pgfplots}
        \usetikzlibrary{intersections}
	\usetikzlibrary{decorations.pathmorphing}
        \usetikzlibrary{topaths}
        \usetikzlibrary{automata,arrows}
        \usetikzlibrary{shapes.geometric} 
        \usetkzobj{all}
        \usetikzlibrary{backgrounds}

\begin{document}

\begin{abstract}
We consider an orbit category of the bounded derived category of a path algebra of type $A_n$ which can be viewed as a $-(m+1)$-cluster category, for $m\geq 1$. In particular, we give a characterisation of those maximal $m$-rigid objects whose endomorphism algebras are connected, and then use it to explicitly study these algebras. Specifically, we give a full description of them in terms of quivers and relations, and relate them with (higher) cluster-tilted algebras of type $A$. As a by-product, we introduce a larger class of algebras, called {\it tiling algebras}.
\end{abstract}



\maketitle

%============================================================================
% Introduction
\addtocontents{toc}{\protect{\setcounter{tocdepth}{-1}}}  % No toc entry for Introduction
\section*{Introduction} 
\addtocontents{toc}{\protect{\setcounter{tocdepth}{1}}}   % but enable toc entries for other sections
%============================================================================

In \cite{BMRRT}, and independently in \cite{CCS} in type A, the authors introduced cluster categories as a categorical model of cluster algebras. These categories led to the development of so-called cluster-tilting theory, an important generalisation of classical tilting theory, and are therefore of central importance in representation theory.

The cluster category of an acyclic quiver $Q$ is defined to be the $\tau^{-1} \Sigma$-orbit category of the bounded derived category $\sD^b (\kk Q)$ of the corresponding path algebra $\kk Q$, where $\tau$ is the Auslander-Reiten translate and $\Sigma$ is the suspension functor in $\sD^b (\kk Q)$. In \cite{Thomas}, Thomas introduced a generalisation of cluster categories, the so-called $m$-cluster categories, for $m \geq 1$. These are the $\tau^{-1} \Sigma^m$-orbit categories of $\sD^b (\kk Q)$. A key property of these categories, which inspired several authors to study cluster-tilting theory in a more general set-up, is that they are $(m+1)$-Calabi-Yau triangulated categories. 

In the present article, we will consider the orbit categories $\sB_m (A_n)$ of $\sD^b (\kk A_n)$ by $\tau \Sigma^{m+1}$, for $m \geq 1$, where $\kk A_n$ is a path algebra of type $A_n$. These categories can be viewed as $-(m+1)$-cluster categories and $(-m)$-Calabi-Yau ($\Sigma^{-m}$ is the Serre functor). In particular, they can be considered to be of negative Calabi-Yau `dimension'. Further reasons to support this idea can be found in \cite{CS13} and \cite{CSP}.

The main interest in $m$-cluster categories, and other positive CY-triangulated  categories, has arisen from the nice homological and combinatorial properties of $m$-cluster-tilting objects and their corresponding endomorphism algebras. In the acyclic case, $m$-cluster-tilting objects coincide with maximal $m$-rigid objects: objects which are maximal with respect to the property that $\Ext^i (t,t^\prime) = 0$, for every pair of summands $t, t^\prime$ and for all $1 \leq i \leq m$. 

In this article, we will study a subclass of the maximal $m$-rigid objects of $\sB_m (A_n)$, namely those whose endomorphism algebras are connected. We note that, when $m = 1$, this is the whole class of maximal $m$-rigid objects.

Cluster-tilting theory has also led to an interest in associating combinatorial models to triangulated categories, in particular the (higher-) cluster categories. These models facilitate the provision of simple characterisations of several representation-theoretic objects, which are more tractable than the objects themselves. For instance, $m$-cluster-tilting objects in type $A_n$ can be simply described via $(m+2)$-angulations of an $(m (n+1)+2)$-gon. Note that when at least one of the $(m+2)$-gons in an $(m+2)$-angulation has two disjoint boundary segments, the corresponding $m$-cluster-tilted algebra is disconnected. 

We will also make use of a combinatorial model for $\sB_m (A_n)$. This model, in which indecomposable objects correspond to the `$(m+1)$-diagonals' of an $((m+1)(n+1)-2)$-gon was introduced in \cite{CS13}. We note that, in the case when $m=1$, maximal rigid objects were characterised in \cite{CS11} using a different combinatorial model, in terms of oriented diagonals in an $n$-gon. The characterisation (Theorem~\ref{thm:maximalmrigid}) we present in this article is not only more general, it also has a simpler description, which enables us to develop a deeper understanding of the corresponding endomorphism algebras. 

The collection of $(m+1)$-diagonals (viewed inside the polygon) corresponding to a maximal $m$-rigid object in $\sB_m (A_n)$ can `admit crossings' or `contain regions bounded by disjoint boundary segments'; in which case the associated endomorphism algebra is disconnected. The behaviour of such collections of $(m+1)$-diagonals is not as neat, and so we will restrict to the connected case. The connected endomorphism algebras can be realised as {\it tiling algebras}, whose notion we introduce in this paper. Tiling algebras arise from {\it tilings} of a disc, which can be seen as polygon dissections, where each subpolygon is referred to as a {\it tile}. 

We will prove that tiling algebras are precisely those gentle algebras for which cycles are oriented and full of relations (Proposition \ref{prop:tilinggentle}). We will then give a complete description of the endomorphism algebras of connected maximal $m$-rigid objects in $\sB_m(A_n)$ (Theorem~\ref{thm:endalg}), by giving a list of conditions on the so-called `permitted' and `forbidden' paths of their quivers. These results make use of a tool we call the {\it tiling algorithm} which associates a tiling to a gentle algebra whose cycles are oriented and full of relations. It is interesting to note that this algorithm can be seen as a `dual' version of an algorithm given in \cite{S14}, which associates a Brauer graph to a gentle algebra. 

We would also like to point out that a tiling can be viewed as a partial triangulation of a polygon. Consequently, the tiling algebras, which generalise the notion of surface algebras (in the case of a disc) introduced in \cite{DRS}, are precisely the endomorphism algebras of partial cluster-tilting objects of type $A$. Other algebras associated to partial triangulations have recently been studied in \cite{Demonet}. 

The $k$-cluster-tilted algebras, $k \geq 1$, of type $A$ are an obvious example of tiling algebras. We examine the relationship between them and the connected endomorphism algebras of maximal $m$-rigid objects ($m \geq 1$) in $\sB_m (A_n)$ in Section~\ref{sec:clustertilted}. Then, to illustrate how amenable to computation tiling algebras are, we finish this article by computing the Gorenstein dimension and the AG-invariant (which gives a necessary condition for derived equivalence) of an arbitrary tiling algebra.

%=============================================================================
%Section
\section{Background} \label{sec:background}
%=============================================================================

Let $\kk$ be an algebraically closed field, $n, m \in \bN$ and $\Db (\kk A_n)$ be the bounded derived category of a path algebra of type $A_n$. Let $\tau$ be the AR-translate on $\Db (\kk A_n)$ and $\Sigma$ the shift functor. In this paper we will consider the orbit category $\sB_m (A_n)$ of $\Db (\kk A_n)$ by $\tau \Sigma^{m+1}$. 

\subsection{A combinatorial model for $\sB_m (A_n)$.} In \cite{CS13} a combinatorial model for $\sB_m (A_n)$, which will be used in this paper, was introduced. 

We shall now recall the description of the combinatorial model in \cite{CS13}, for the convenience of the reader. Let $\mcP_{n,m}$ be the regular $((m+1)(n+1)-2)$-gon, with vertices numbered clockwise from $1$ to $(m+1)(n+1)-2$. All operations on vertices of $\mcP_{n,m}$ will be performed modulo $(m+1)(n+1)-2$, with representatives $1, \ldots, (m+1)(n+1)-2$. An \emph{$(m+1)$-diagonal} of $\mcP_{n,m}$ is a diagonal that divides $\mcP_{n,m}$ into two polygons each of whose number of vertices is divisible by $m+1$. An $(m+1)$-diagonal linking vertices $i, j$ of $\mcP_{n,m}$ will be denoted by $\{i,j\}$. The $(m+1)$-diagonals of the form $\{i,i+m\}$ are called {\it short diagonals}. The non-short diagonals are called {\it long diagonals}.

Let $\Gamma (n,m)$ be the stable translation quiver whose vertices are the $(m+1)$-diagonals of $\mcP_{n,m}$ and arrows are obtained in the following way: given two $(m+1)$-diagonals $D$ and $D'$ with a vertex $i$ in common, there is an arrow from $D$ to $D'$ in $\Gamma (n,m)$ if and only if $D'$ can be obtained from $D$ by rotating clockwise $m+1$ steps around $i$. The translation automorphism $\tau: \Gamma (n,m) \to \Gamma (n,m)$ sends an $(m+1)$-diagonal $\{i,j\}$ to $\tau (\{i,j\}):= \{i-m-1, j-m-1\}$.

The AR quiver of $\sB_m (A_n)$ is equivalent to the stable translation quiver $\Gamma (n,m)$. Note that $\Sigma \{i,j\} = \{i+1,j+1\}$. 

\begin{example}
The category $\sB_2 (A_3)$ is equivalent to $\Gamma (3,2)$, which is as follows:
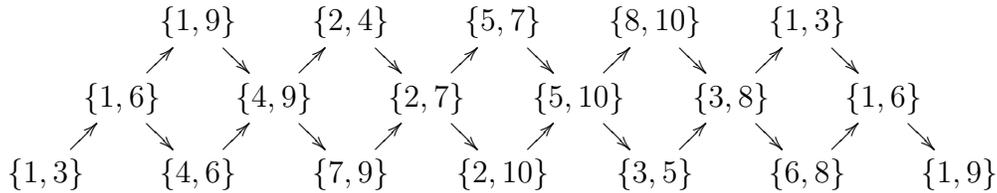
\[
\xymatrix@!=0.4pc{&                       & \{1,9\} \ar[dr]         &                       & \{2,4\} \ar[dr]       &                      & \{5,7\} \ar[dr]
&                      & \{8,10\} \ar[dr]       &                     & \{1,3\} \ar[dr] \\
                      & \{1,6\} \ar[dr]\ar[ur]  &                       & \{4,9\} \ar[dr]\ar[ur] &                      & \{2,7\} \ar[dr]\ar[ur] &
& \{5,10\} \ar[dr]\ar[ur] &                     & \{3,8\} \ar[dr]\ar[ur] &                          & \{1,6\} \ar[dr] \\
\{1,3\} \ar[ur] &                       & \{4,6\} \ar[ur]  &                       & \{7,9\} \ar[ur] &                      & \{2,10\} \ar[ur]
&                      & \{3,5\} \ar[ur] &                     & \{6,8\} \ar[ur] &                     & \{1,9\}}
\]
\end{example}

We shall identify (isoclasses of) indecomposable ojects in $\sB_m (A_n)$ with the corresponding $(m+1)$-diagonals, and we shall freely switch between objects and diagonals. We will use Roman typeface for the indecomposable objects of $\sB_m (A_n)$ and typewriter typeface for the corresponding diagonals.

\subsection{Combinatorial description of the Ext-hammocks.} The main objects of study in this paper are {\it maximal $m$-rigid objects}, that is, basic objects with a maximal number of indecomposable direct summands for which $\Ext^k_{\sB_m (A_n)} (x,y) = 0$, for any pair of indecomposable summands $x$ and $y$, and $k \in \{1, \ldots, m\}$. Lemma \ref{lemma:exthammockconditions}, which describes the (forward) $\Ext^i$-hammock of an indecomposable object in terms of $(m+1)$-diagonals, will be essential to our characterisation of these objects. 

Given $k \in \{ 1, \ldots, m \}$, we say that two $(m+1)$-diagonals $\arc{a}$ and $\arc{b}$ are {\it $k$-neighbours} provided they do not cross and there is some vertex $v$ incident with $\arc{a}$ such that $\arc{b}$ is incident with $v\pm k$. 

The following notation will be useful throughout the paper. 

\begin{notation}
\begin{compactenum}
\item Given the vertices $i_1, i_2, \ldots, i_k$ of $\mcP_{n}$, we write $C(i_1, i_2, \ldots, i_k)$ to mean that $i_1, i_2, \ldots, i_k, i_1$ follow each other under the clockwise circular order on the boundary of $\mcP_{n}$.
\item Given two vertices $i, j$ of $\mcP_{n,m}$, we define $[i,j]$ to be the number of vertices encountered when travelling along the boundary in the clockwise direction from $i$ to $j$, inclusive.
\end{compactenum}
\end{notation}

Note that $[i,j]$ is a multiple of $m+1$ if and only if $\{i,j\}$ is an $(m+1)$-diagonal.

\begin{lemma}\label{lemma:exthammockconditions}
Let $k \in \{1, \ldots, m\}$, $a, b \in \ind{\sB_m (A_n)}$ and $\arc{a} = \{a_1, a_2\}$, with $a_1 < a_2$, be the $(m+1)$-diagonal corresponding to $a$. Then $\Ext^k_{\sB_m (A_n)} (b,a) \ne 0$ if and only if $\arc{b}$ satisfies one of the following conditions:
\begin{compactenum}
\item $\arc{b}$ is a $k$-neighbour of $\arc{a}$ incident with $a_2 +k$,
\item $\arc{b}$ is a $k$-neighbour of $\arc{a}$ incident with $a_1 +k$, 
\item $\arc{b}$ crosses $\arc{a}$ in such a way that $\arc{b} = \{b_1, b_2\}$ with $C(a_1,b_1,a_2,b_2)$ and $[a_i,b_i] = x_i (m+1) + k$, for some $x_i \geq 1$, $i = 1, 2$. 
\item $\arc{b} = \{a_1+k,a_2+k\}$, i.e. $b = \Sigma^k a$. 
\end{compactenum} 
\end{lemma}
\begin{proof}
The case $k = 1$ is \cite[Corollary 10.6]{CSP}. For $k \geq 2$, we prove the result by applying \cite[Corollary 10.6]{CSP} to $b$ and $\Sigma^{k-1} a$. We must first check that $\arc{b}$ is a $1$-neighbour of $\Sigma^{k-1} \arc{a}$ if and only if $\arc{b}$ is a $k$-neighbour of $\arc{a}$. Suppose that $\arc{b} = \{b_1, b_2\}$ is a $1$-neighbour of $\Sigma^{k-1} \arc{a}$. We may assume that $b_1 = a_1 + k$. Then either $\arc{b}$ is a $k$-neighbour of $\arc{a}$ or $b_2 = a_2 +j$, for some $0 \leq j < k-1$. However, the latter would imply that $\arc{b}$ is not an $(m+1)$-diagonal, since $k < m+1$. Therefore, $\arc{b}$ is a $k$-neighbour of $\arc{a}$. The converse is trivial.

Now we must show that $\arc{b}$ crosses $\Sigma^{k-1} \arc{a}$ such that $[a_i + k-1,b_i] = x_i (m+1) +1$, for $i= 1, 2$ and for some $x_i \geq 1$ if and only if $\arc{b}$ crosses $\arc{a}$ such that $[a_i,b_i] = x_i (m+1) +k$, for $i = 1, 2$.

Suppose $\arc{b}$ crosses $\Sigma^{k-1} \arc{a}$ such that $[a_i + k-1,b_i] = x_i (m+1) +1$, for $i= 1, 2$ and for some $x_i \geq 1$. Since $\arc{b}$ is an $(m+1)$ diagonal, we have that $[b_1,b_2] = \ell (m+1)$, for some $\ell \geq 1$. It follows from the fact that $\arc{b}$ and $\Sigma^{k-1} \arc{a}$ cross that $\ell > k_1$ and $n+1-\ell > k_2$. Therefore, $[b_2+1,a_1+k-2] = (\ell -k_1) (m+1)-2 > m-1 \geq k-1$. Likewise, $[b_1 +1,a_2+k-2] > k-1$. Therefore, $\arc{b}$ crosses $\arc{a}$ and  $[a_i,b_i] = x_i (m+1) +k$, for $i = 1, 2$. The converse is similar.
\end{proof}

Figure \ref{fig:ext-shaded} shows where the indecomposable objects corresponding to the arcs that satisfy the conditions in Lemma \ref{lemma:exthammockconditions} lie in the AR quiver.

\begin{figure}[!ht]

\begin{center}
\begin{tikzpicture}[thick,scale=0.75, every node/.style={scale=0.75}]

%a:
\fill (-10.5,0.5) circle (2pt);
\node (a) at (-10.5,0.5) {};
\node [left] at (a) {$\Sigma^{k-1} a$};

%shaded areas
\filldraw[fill=black!20!white] (-10,0.5) -- (-8.5,2) -- (-6,-0.5) -- (-7.5,-2) -- (-10,0.5); %case 3

\filldraw[fill=black!30!white] (-8.5,2) -- (-8,2.5) -- (-5.5,0) -- (-6,-0.5) -- (-8.5,2); %case 2

\filldraw[fill=black!40!white] (-7.5,-2) -- (-6,-0.5) -- (-5.5,-1) -- (-7,-2.5) -- (-7.5,-2); %case 1

\filldraw[fill=black!10!white] (-6,-0.5) -- (-5.5,0) -- (-5,-0.5) -- (-5.5,-1) -- (-6,-0.5); %case 4

%tauinverse of a
\fill[black!50!white] (-9.5,0.5) circle (2pt);
\node (tauinv-a) at (-9.5,0.5) {};
\node [right] at (tauinv-a) {$\tau^{-1} \Sigma^{k-1} a$};

%hom-hammock of tau-inverse of a:
\draw[dotted] (-9.5,0.5) -- (-8,2) -- (-5.5,-0.5) -- (-7,-2) -- (-9.5,0.5);

%shaded areas notation:
\node (2) at (-6.5,1.5) {};
\node at (2) {$(2)$};
\node (3) at (-9,-1) {};
\node at (3) {$(3)$}; 
\node (4) at (-5,-0.5) {};
\node[right] at (4) {$(4)$};
\node (1) at (-5.75,-2){};
\node[above] at (1) {$(1)$};

%top and bottom lines of the AR-quiver:
\draw[thick] (-11,-2) -- (-7.5,-2);
\draw[dotted] (-7.5,-2) -- (-6.5,-2);
\draw[thick] (-6.5,-2) -- (-5,-2);

\draw[thick] (-11,2) -- (-8.5,2);
\draw[dotted] (-8.5,2) -- (-7.5,2);
\draw[thick] (-7.5,2) -- (-5,2);
\end{tikzpicture}
\end{center}
\caption{$\Ext^k_{\sB_m (A_n)}(-,a) \neq 0$.}
\label{fig:ext-shaded}
\end{figure}

\begin{remark}
It follows immediately from Lemma \ref{lemma:exthammockconditions} that, for $n \geq 2$, $\Ext^k_{\sB_m (A_n)} (a,a) = 0$, for any object $a \in \ind{\sB_m (A_n)}$ and $1 \leq k \leq m$. In the case when $n=1$, we have $\Ext^m_{\sB_m (A_1)} (a,a) \neq 0$, for any object $a \in \ind{\sB_m (A_1)}$. Therefore, there are no $m$-rigid objects in $\sB_m (A_1)$, and so we take $n \geq 2$ throughout the rest of the paper.  
\end{remark}

%==============================================================================
%  Section
\section{Classification of connected maximal $m$-rigid objects} \label{sec:classification}
%==============================================================================

The aim of this section is to give a characterisation of a large subclass of the maximal $m$-rigid objects of $\sB_m (A_n)$, which we call \emph{connected} maximal $m$-rigid objects, and will be defined below. This characterisation will make use of the geometric model described in Section~\ref{sec:background}.  

We can view a maximal $m$-rigid object $T$ as a graph whose vertices are the vertices of $\mcP_{n,m}$ and whose edges are (correspond to) the indecomposable summands of $T$. We will determine the properties that characterise these graphs. The method used is similar to the one used in \cite{CS11}, but we will see that this geometric model provides a much neater description of these objects. 

From now on, we shall tacitly switch between interpreting a maximal $m$-rigid object as a direct sum of indecomposable objects in $\sB_m (A_n)$ and as a graph whose vertices are those of $\mcP_{n,m}$, as described.

\begin{proposition}\label{prop:m-rigidconditions}
A set of objects $\sT$ in $\sB_m (A_n)$ is $m$-rigid if and only if it satisfies the following conditions:
\begin{compactenum}
\item There are no $k$-neighbours, for every $k \in \{1, \ldots, m\}$. 
\item There are no adjacent short diagonals.
\item The only possible crossings are between a short diagonal and a long diagonal. 
\end{compactenum}
\end{proposition}
\begin{proof}
Suppose $\sT$ is an $m$-rigid object. Given an arc $\arc{a} = \{a_1, a_2\}$ in $\sT$, $k \in \{1, \ldots m\}$ and a $k$-neighbour $\arc{b}$ of $\arc{a}$, then $\arc{b}$ must be incident with one of the following vertices: $a_1+k, a_2 +k, a_1-k$ or $a_2 - k$. In the first two cases, we have $\Ext^k (b,a) \neq 0$ and in the other cases, we have $\Ext^k (a,b) \neq 0$, by Lemma \ref{lemma:exthammockconditions}. Therefore, $\arc{b} \not\in \sT$, and (1) holds. Now, given a short diagonal $\arc{a} = \{a_1, a_1+m\}$, we have $\Sigma^m \arc{a} = \{a_1 + m, a_1 + 2m\}$ and $\Sigma^{-m} \arc{a} = \{a_1 -m, a_1\}$. Therefore, adjacent short diagonals have an $\Ext^m$, and so (2) holds. 

Now, let $\arc{a}, \arc{b}$ be two short arcs in $\sT$ which cross. Assume, without loss of generality, that $C(a_1,b_1,a_2,b_2)$. Then there is $1 \leq k \leq m-1$ such that $\arc{b} = \Sigma^k \arc{a}$, and so $\Ext^k (b,a) \neq 0$, a contradiction. Suppose now, $\arc{a}, \arc{b}$ are two long arcs in $\sT$ which cross. Assume, without loss of generality, that $C(a_1,b_1,a_2,b_2)$. Since $\arc{a}$ and $\arc{b}$ are long $(m+1)$-diagonals, we can write $[a_1,a_2] = t(m+1)$, and $[b_1,b_2] = t' (m+1)$, for some $1 < t, t' < n$. 

\Case{$[a_1+1,b_1] =  x$, for some $1 \leq x \leq m$} We have $[b_1+1,a_2] = [a_1,a_2] - [a_1,b_1] = t (m+1) - (x+1) = (t-1) (m+1) + (m-x)$, where $0 \leq m-x \leq m-1$ and $t-1 \geq 1$, since $t>1$. Similarly, $[b_2+1,a_1] = [b_2,b_1]-[a_1, b_1] = (n+1-t') (m+1) -(x+1) = (n-t') (m+1) + (m-x)$, with $0 \leq m-x \leq m-1$ and $n-t' \geq 1$. Hence, by Lemma \ref{lemma:exthammockconditions}, $\Ext^{m-x+1} (a,b) \neq 0$, a contradiction.  

\Case{$[a_1+1,b_1] = \ell (m+1) + i$, for some $\ell \geq 1$ and $0 \leq i \leq m-1$} In this case, we must have $[a_2 +1, b_2] = \ell^\prime (m+1) +i$, for some $\ell^\prime \geq 0$. If $\ell^\prime \geq 1$, then $\Ext^{i+1} (b,a) \neq 0$ by Lemma \ref{lemma:exthammockconditions}, whereas $\ell^\prime = 0$ is covered by the previous case. 

\Case{$[a_1+1,b_1] = \ell (m+1) + m$, for some $\ell \geq 1$} In this case, we must have $[b_2+1, a_1] =  y_1 (m+1)$ and $[b_1+1,a_2] =  y_2 (m+1)$, for some $y_1, y_2 \geq 1$. Hence, $\Ext^1 (a,b) \neq 0$, a contradiction. This implies that (3) must hold. 

Conversely, let $\sT$ be a set of objects in $\sB_m (A_n)$ satisfying conditions (1), (2) and (3). If $\arc{a} = \{a_1, a_2\}$ and $\arc{b} = \{b_1,b_2\}$ are $(m+1)$-diagonals in $\sT$ which cross, then one of the arcs, say $\arc{a}$, must be short and the other one long. Suppose, without loss of generality that $C(a_1,b_1,a_2,b_2)$ and $[a_1,b_1] = x$ and $[b_1,a_2] = y$, for some $1 \leq x, y \leq m$, since $\arc{a}$ is short. But then there is no $1 \leq k \leq m$ for which this crossing satisfies condition (3) in Lemma \ref{lemma:exthammockconditions}. It is also easy to see that the other conditions in Lemma \ref{lemma:exthammockconditions} (for every $1 \leq k \leq m$) are not satisfied either, which implies that $\sT$ is $m$-rigid.
\end{proof}

\begin{remark}
If $m=1$, no $2$-diagonal can cross a short $2$-diagonal. Therefore, there are no crossings in a maximal $1$-rigid object in $\sB_1(A_n)$. 
\end{remark}

\begin{lemma}\label{lem:isos-shortdiagonals}
Let $\sT$ be a maximal $m$-rigid object in $\sB_m (A_n)$ and $\arc{a} = \{a_1, a_1 + m\}$ a short arc in $\sT$. Then the vertices $a_1 +m + i$ and $a_1 -i$, for $1 \leq i \leq m$, are isolated vertices. 
\end{lemma}
\begin{proof}
This is a straightforward consequence of the fact that there are no $k$-neighbours, for $1 \leq k \leq m$, no crossings between short arcs, and no adjacent short arcs.  
\end{proof}

\begin{definition}
Let $T$ be a maximal $m$-rigid object of $\sB_m (A_n)$ and $\sT$ be the corresponding graph in $\mcP_{n,m}$. We say that $T$ is \textit{connected} if the full subgraph of $\sT$ with set of vertices given by the non-isolated vertices in $\sT$ is connected. 
\end{definition}

Our aim is to classify the connected maximal $m$-rigid objects of $\sB_m (A_n)$, that is, the maximal $m$-rigid objects which are also connected. We will see in Section~\ref{sec:endoalgebras} that these are precisely the maximal $m$-rigid objects for which the corresponding endomorphism algebras are connected. In the case when $m=1$, every maximal $1$-rigid object is connected (see \cite{CS11}).

The following lemma states that there will be no crossings in connected maximal $m$-rigid objects.  

\begin{lemma}
\label{lem:crossdisc}
Let $T$ be a maximal $m$-rigid object in $\sB_m (A_n)$ with corresponding graph $\sT$ in $\mcP_{n,m}$, and let $\sT^\prime$ be the full subgraph of $\sT$ on the non-isolated vertices. If $\sT^\prime$ contains a crossing, then it is disconnected.
\end{lemma}
\begin{proof}
By Proposition \ref{prop:m-rigidconditions}, a crossing must be between a short arc $\arc{a}$ and a long arc $\arc{b}$. Since any other arc incident with one of the endpoints of $\arc{a}$ is either $\Sigma^m \arc{a}$, $\Sigma^{-m} \arc{a}$, a long arc crossing $\arc{b}$ or a $k$-neighbour of $\arc{b}$, for some $1 \leq k \leq m-1$, we must have that the valency of the endpoints of $\arc{a}$ is one. Therefore, the graph is disconnected. 
\end{proof}

A collection of noncrossing diagonals in a marked disc $\mcP$ induces a dissection of $\mcP$. We call each region in this dissection a {\it tile}. In particular, each region of $\mcP_{n,m}$ in the graph $\sT$ corresponding to a connected maximal $m$-rigid object $T$ is a tile. We will now examine the possible tiles appearing in these graphs. We can divide the tiles in two types: with or without open boundary.

\begin{definitions}
Let $\sG$ be a noncrossing unoriented graph in $\mcP_{n,m}$, $\mcT$ a tile of $\sG$ and $i_1, i_2, \ldots, i_k$ be $k$ consecutive vertices of $\mcP_{n,m}$ that lie in $\mcT$, with $k \geq m+1$.
\begin{compactenum}
\item If $k = m+1$, we say that $\mcT$ has an \textit{open boundary} $(i_1,i_{m+1})$ if $\mcT$ is bounded by the short $(m+1)$-diagonal $\{i_1,i_{m+1}\}$, and $i_2, \ldots, i_m$ are isolated vertices.  
\item If $k >m+1$, we say that $\mcT$ has an \textit{open boundary} $(i_1, i_k)$ if $i_2, \ldots, i_{k-1}$ are isolated, and $i_1, i_k$ are incident with (bounding) arcs of $\mcT$.
\item The \textit{length of the open boundary} $(i_1, i_k)$, with $k \geq m+1$, is defined to be $k-1$.
\item The \textit{length of the tile $\mcT$} is given by the number of $(m+1)$-diagonals which bound $\mcT$, and it is denoted by $\ell (\mcT)$.
\end{compactenum}
\end{definitions}

Given a tile $\mcT$ of a connected maximal $m$-rigid object with an open boundary $(i_1,i_k)$, it is clear that the vertices $i_2, \ldots, i_{k-1}$ are isolated. Moreover, any given tile has at most one open boundary, by the connectedness assumption. 

\begin{lemma}\label{lemma:cycles}
Let $T$ be a connected maximal $m$-rigid object in $\sB_m (A_n)$, and $\sT$ the corresponding graph in $\mcP_{n,m}$. 
\begin{compactenum}
\item No cycle in $\sT$ contains a short diagonal.
\item Any simple cycle in $\sT$ has length $m+3$.
\end{compactenum}
\end{lemma}
\begin{proof}
(1) If a cycle in $\sT$ contained a short diagonal, then there would be $m$-neighbouring $(m+1)$-diagonals, which is a contradiction. 

(2) Let $\mcC$ be a cycle in $\sT$ of length $k$, and let $v_1, v_2, \ldots, v_k$ be the vertices of $\mcC$ such that $C(v_1, v_2, \ldots, v_k)$. Since $\{v_k,v_1\}$ is an $(m+1)$-diagonal, we must have $[v_1,v_k] = x (m+1)$, for some $x \geq 1$. But $[v_1,v_k] = \sum\limits_{i=1}^{k-1} [v_i,v_{i+1}] - (k-2) = \sum\limits_{i=1}^{k-1} x_i (m+1) - (k-2)$, for some $x_i \geq 1$. Therefore, $k-2$ must be of the form $y(m+1)$, and hence $k = y(m+1)+2$, for some $y \geq 1$. 

Suppose now $\mcC$ is simple, and assume $y \geq 2$. Consider $\{v_1, v_{m+3}\}$, which is a long $(m+1)$-diagonal, and does not lie in $\sT$. Since $\mcC$ is simple, and $T$ is an $m$-rigid object for which there are no crossings, we have that $\{v_1, v_{m+3}\}$ does not cross any diagonal in $\sT$. Given $k = 1, \ldots, m$, any $k$-neighbour of $\{v_1,v_{m+3}\}$ in $\sT$ would either be a $k$-neighbour of $\{v_1,v_2\}, \{v_{m+2},v_{m+3}\}, \{v_{m+3},v_{m+4}\}$ or $\{v_{k},v_1\}$ or it would cross one of these $(m+1)$-diagonals which lie in $\sT$. Therefore, $T \oplus \{ \{v_1, v_{m+3}\} \}$ is $m$-rigid, contradicting the maximality of $\sT$. Hence, $y = 1$, and so any simple cycle must have length $m+3$.  
\end{proof}

We shall partition the tiles of a connected maximal $m$-rigid object into types. Each type of tile is completely determined by two numbers: the length of the tile and the length of its open boundary. By convention, if $\mcT$ is a tile with no open boundary, then the length of its open boundary is zero. We can then denote a type $(\mcT)$ of tiles as a pair $(\ell, b)$, where $\ell$ is the length of any tile of type $(\mcT)$ and $b$ is the length of its open boundary.

\begin{remark}\label{rem:diagonalsintiles}
If $\mcT$ has an open boundary, let $v_1, \ldots, v_{k}$ be the set of vertices of $\mcP_{n,m}$ such that $C(v_1, \ldots, v_{k})$, and $\{v_i,v_{i+1}\}$ is an $(m+1)$-diagonal bounding $\mcT$. Then, using the same argument as in Lemma \ref{lemma:cycles} (1), we can see that the diagonals $\{v_i,v_{i+1}\}$, for $2 \leq i \leq k-2$ must be long. 
\end{remark}

We call the $(m+1)$-diagonals $\{v_1,v_2\}$ and $\{v_{k-1},v_{k}\}$ the \emph{outer-diagonals}. 

\begin{proposition}\label{prop:possibletiles}
Let $T$ be a connected maximal $m$-rigid object in $\sB_m (A_n)$. The possible tiles of the corresponding graph $\sT$ lie in one of the following classes: 
\begin{compactenum}
\item[$(\mcT_k)$] $(k, m+k-1)$, where $1 \leq k \leq m+1$. 
\item[$(\mcT_{m+2})$] $(m+2, 2m+1)$, and at least one of the outer-diagonals is short.
\item[$(\mcT'_1)$] $(1, 2m+1)$,
\item[$(\mcT_{m+3})$] $(m+3, m+1)$, and both outer-diagonals are short.
\item[$(\mcT'_{m+3})$] $(m+3, 0)$. 
\end{compactenum}
\end{proposition}
\begin{proof}
We know that each tile $\mcT$ in $\sT$ has either zero or one open boundary, given the connectedness assumption. Suppose $\mcT$ has no open boundary, so that $\mcT$ is a cycle. It follows from Lemma~\ref{lemma:cycles} that any cycle in $\sT$  bounds a region which is a union of tiles, each of which is bounded by an $(m+3)$-cycle. Hence, $\mcT$ must be a tile of type $(\mcT'_{m+3})$. 

Suppose now $\mcT$ has an open boundary with length $b$. The number of isolated vertices in the open boundary is then $b-1$. Let $v_1, \ldots, v_{k}$ be as in Remark \ref{rem:diagonalsintiles}, and order the isolated vertices $i_1, \ldots, i_{b-1}$ in the open boundary such that $C(v_k, i_1, \ldots, i_{b-1}, v_1)$. Note, in particular, that $\mcT$ is of type $(k-1,b)$. 

{\it Claim 1.} We must have $m \leq b \leq 2m+1$. 

Since there are no $i$-neighbours in $\sT$, for any $1 \leq i \leq m$, and a short arc is of the form $\{a,a+m\}$, we cannot have $b-1 \leq m-2$. On the other hand, suppose $b-1 \geq 2m+1$. Then, $\{v_{k-1}, i_{m+1}\}$ is a long $(m+1)$-diagonal which does not lie in $\sT$ and does not introduce any crossings or $i$-neighbours, for $1 \leq i \leq m$, since $b-1 \geq 2m+1$ and $\{v_{k-2},v_{k-1}\}, \{v_{k-1},v_k\} \in \sT$. Hence, $T \oplus \{v_{k-1}, i_{m+1}\}$ is $m$-rigid, contradicting the maximality of $T$. Therefore, $b-1 \leq 2m$ and the claim is proved. 

{\it Claim 2.} We must have $1 \leq k-1 \leq m+3$. 

Suppose, for a contradiction, that $k-1 \geq m+4$. Then $\{v_2, v_{m+4}\}$ is a long $(m+1)$-diagonal which does not lie in $\sT$ and such that $T \oplus \{v_2, v_{m+4}\}$ is $m$-rigid, a contradiction which finishes the proof of this claim. 

Assume $k = 2$, i.e.~$\mcT$ is bounded only by one $(m+1)$-diagonal and one open boundary with $b-1$ isolated vertices. We know, by Claim 1, that $m-1 \leq b-1 \leq 2m$. Hence, since $\{v_1,v_2\}$ is an $(m+1)$-diagonal, we must have either $b = m$, in which case $\mcT$ is of type $(\mcT_1)$, or $b = 2m+1$, in which case $\mcT$ is of type $(\mcT'_1)$.

Now assume $2 \leq k-1 \leq m+2$. We want to check that $b = m+k-2$. Since $\{v_{k-1},v_k\}$ is an $(m+1)$-diagonal, we have $[v_k,v_{k-1}] = x (m+1)$, for some $x \geq 1$. On the other hand, $[v_k, v_{k-1}] = 1 + (b-1) + \sum\limits_{i=1}^{k-2} [v_i,v_{i+1}] - (k-3)$, and so $b = \ell (m+1) + (k-3)$, for some $\ell$. We must have $\ell \geq 1$, otherwise either $b-1 < 0$ or  $\{v_1,v_2\}$ and $\{v_{k-1},v_{k}\}$ would be $m$-neigbours, a contradiction. But if, $\ell \geq 2$, then $b \geq 2m + 2$, contradicting Claim 1. Therefore, $\ell = 1$ and $b = m +k-2$, as required. 

Finally, assume $k-1 = m+3$. Then $b = \ell (m+1) + (k-3) = (\ell+1) (m+1)$, for some $\ell$. If $\ell \geq 1$, then $b \geq 2m +2$, and if $\ell < 0$, then $b \leq 0$, a contradiction. Therefore, $\ell = 0$, and $b= m+1$.    

Suppose $k-1 = m+2$. We need to check that either $\{v_1,v_2\}$ or $\{v_{m+2},v_{m+3}\}$ is short. Suppose neither of them is. Then $\{v_1,v_{m+3}\}$ is a long $(m+1)$-diagonal which does not lie in $\sT$ and such that $T \oplus \{v_1,v_{m+3}\}$ is $m$-rigid, contradicting the maximality of $T$. 

Finally let $k-1 = m+3$. If $\{v_1,v_2\}$ were a long diagonal, then $\{v_1,v_{m+3}\}$ is a long $(m+1)$-diagonal such that it does not lie in $\sT$ and $T \oplus \{v_1,v_{m+3}\}$ is $m$-rigid, contradicting maximality of $T$. Similarly, if $\{v_{m+3},v_{m+4}\}$ were a long diagonal, we would have $T \oplus \{v_2, v_{m+4}\}$ $m$-rigid, a contradiction. Hence, both diagonals $\{v_1, v_2\}$ and $\{v_{m+3},v_{m+4}\}$ are short. 
\end{proof}

\begin{definition}
Given a disc $\mcP$ with marked points on the boundary, a {\it tiling of $\mcP$} is defined to be a collection of tiles glued to each other in such a way that they cover the polygon $\mcP$. 
\end{definition}

\begin{example}
Figure \ref{fig:tilesm=1} illustrates a tiling of $\mcP_{10,1}$. This tiling has examples of all possible tile types that arise when $m = 1$. The light shaded tiles are of type $\mcT_1$ and the dark shaded tiles are of type $\mcT_1^\prime$. 
\begin{figure}[!ht]
\begin{center}
\begin{tikzpicture}[thick,scale=0.6, every node/.style={scale=0.6}]

\draw (0,0) circle (3cm);

%tiles of type T_1

\filldraw[fill=black!10!white] (18:3cm) .. controls (27.5:2.5cm) .. (36:3cm);

\filldraw[fill=black!10!white] (144:3cm) .. controls (153.5:2.5cm) .. (162:3cm);

\filldraw[fill=black!10!white] (216:3cm) .. controls (225.5:2.5cm) .. (234:3cm);

\filldraw[fill=black!10!white] (270:3cm) .. controls (279.5:2.5cm) .. (288:3cm);

%tiles of type T'_1

\fill[fill=black!30!white] (36:3cm) -- (54:2.5cm) -- (72:2.5cm) -- (90:3cm) -- (72:3cm) -- (54:3cm) -- (36:3cm);
\draw (36:3cm) .. controls (54:2.5cm) and (72:2.5cm) .. (90:3cm);

\fill[fill=black!30!white] (162:3cm) -- (180:2.5cm) -- (198:2.5cm) -- (216:3cm) -- (198:3cm) -- (180:3cm) -- (162:3cm);
\draw (162:3cm) .. controls (180:2.5cm) and (198:2.5cm) .. (216:3cm);

\fill[fill=black!30!white] (288:3cm) -- (306:2.5cm) -- (324:2.5cm) -- (342:3cm) -- (324:3cm) -- (306:3cm) -- (288:3cm);
\draw (288:3cm) .. controls (306:2.5cm) and (324:2.5cm) .. (342:3cm);

\draw (36:3cm) .. controls (72:1cm) and (126:1cm) .. (162:3cm);

\draw (216:3cm) .. controls (270:1cm) and (306:1cm) .. (342:3cm);

\draw (342:3cm) .. controls (9.5:1.5cm) .. (36:3cm);

\node (a) at (0,0) {};
\node at (a) {$\mcT^\prime_{m+3}$};

\node (b) at (-0.3,1.7) {};
\node at (b) {$\mcT_{m+2}$};

\node (c) at (-0.2,-2) {};
\node at (c) {$\mcT_{m+3}$};

\node (d) at (2.3,0.4) {};
\node at (d) {$\mcT_2$};

\foreach \angle in {0,18,36,54,72,90,108,126,144,162,180,198,216,234,252,270,288,306,324,342}
{\filldraw[fill=black] (\angle:3cm) circle (0.07cm);}
\end{tikzpicture}
\end{center}
\caption{Possible tiles when $m = 1$.}
\label{fig:tilesm=1}
\end{figure}
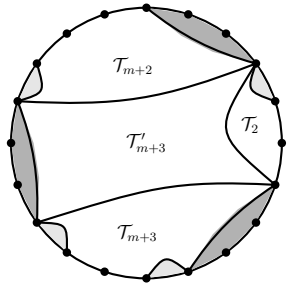
\end{example}

It follows from the connectedness assumption that the noncrossing graph associated to a connected maximal $m$-rigid object corresponds to a tiling of $\mcP_{n,m}$ whose tiles are described in Proposition \ref{prop:possibletiles}. However, not every such tiling corresponds to a maximal $m$-rigid object, as it might not be $m$-rigid or it might not be maximal.  

\begin{lemma}\label{lemma:isolatedverticescondition}
Let $T$ be a connected maximal $m$-rigid object in $\sB_m (A_n)$. If there is a tile with open boundary $(i_0,i_{2m+1})$ of length $2m+1$ then the short diagonals $\{i_0-m, i_0\}$ and $\{i_{2m+1}, i_{2m+1}+m\}$ must lie in $\sT$.
\end{lemma}
\begin{proof}
Denote by $i_1, i_2, \ldots, i_{2m}$ the vertices of $\mcP_{n,m}$ such that $C(i_0, i_1, i_2, \ldots, i_{2m}, i_{2m+1})$. Recall that $i_1, \ldots, i_{2m}$ must be isolated vertices. Suppose that $\{i_0-m,i_0\}$ does not lie in $\sT$. Then $\{i_0,i_m\}$ can be added to $\sT$ without adding crossings, $k$-neighbours for $1 \leq k \leq m$, nor adjacent short diagonals. In other words, $T \oplus \{i_0,i_m\}$ is $m$-rigid, contradicting the maximality of $T$. Similarly, $\{i_{2m+1},i_{2m+1}+m\}$ must also lie in $\sT$.
\end{proof}

We are now ready to state the main result of this section. 

\begin{theorem}\label{thm:maximalmrigid}
There is a one-to-one correspondence between connected maximal $m$-rigid objects in $\sB_m (A_n)$ and tilings of $\mcP_{n,m}$ with tiles as in Proposition \ref{prop:possibletiles} satisfying the following conditions:
\begin{compactenum}
\item If $(i_0, i_{2m+1})$ is an open boundary of length $2m+1$, then $\{i_0-m,i_0\}$ and $\{i_{2m+1},i_{2m+1} + m\}$ are short diagonals belonging to the tiling,
\item There are no adjacent short diagonals. 
\item There is no sequence of consecutive vertices $v_1, \ldots, v_k, v_{k+1}, \ldots, v_\ell$ such that  $C(v_1, \ldots, v_k, v_{k+1}, \ldots, v_\ell)$, $v_1, \ldots, v_{k-1}, v_{k+1}, \ldots, v_\ell$ are isolated, $k-1, \ell-k \geq m+1$ and $\ell \geq 3m+1$.
\end{compactenum}
\end{theorem}
\begin{proof}
If $T$ is a connected maximal $m$-rigid object, then it follows from Proposition \ref{prop:m-rigidconditions} $(2)$, Proposition \ref{prop:possibletiles}, and Lemma \ref{lemma:isolatedverticescondition}, that the corresponding graph is a tiling with tiles as in Proposition \ref{prop:possibletiles} satisfying conditions $(1)$ and $(2)$. If there is a sequence of consecutive vertices satisfying $(3)$, then it would be possible to add a short diagonal $\arc{a}$ which crosses every diagonal incident with $v_k$, and preserve $m$-rigidity, contradicting the maximality of $T$.

Now consider a tiling $\sT$ of $\mcP_{n,m}$ with tiles as in Proposition \ref{prop:possibletiles} satisfying conditions $(1)$, $(2)$ and $(3)$. Then clearly, there are no crossings, no $k$-neighbouring diagonals, for $1 \leq k \leq m$, and no adjacent short diagonals. Therefore, the direct sum of the indecomposable objects corresponding to the edges of the tiling is an $m$-rigid object $T$. 

Suppose there is an indecomposable object $a$, not isomorphic to any summand of $T$ such that $T \oplus a$ is $m$-rigid. 

\Case{$\arc{a}$ is a long diagonal} By Proposition \ref{prop:m-rigidconditions} $(3)$, $\arc{a}$ cannot cross any long diagonal in the tiling. On the other hand, due to the connectedness of the tiling $\sT$ and the neighbouring condition, $\arc{a}$ can only cross a short diagonal $\arc{s}$ if $\arc{s}$ is the only diagonal in $\sT$. By Proposition \ref{prop:possibletiles}, we must have $n=2$. However, there are no long diagonals in this case. Therefore, $\arc{a}$ can only be added to the interior of a tile. However, it is easy to check that the only diagonals that can be added in the interior of the tiles in Proposition \ref{prop:possibletiles}, whilst preserving $m$-rigidity, are short diagonals in tiles of type $(\mcT_1^\prime)$, turning them into tiles of type $(\mcT_2)$, or short diagonals in tiles of type $(\mcT_{m+2})$, turning them into tiles of type $(\mcT_{m+3})$. In particular, long diagonals cannot be added in the interior of tiles of the types described in Proposition \ref{prop:possibletiles}.

\Case{$\arc{a}$ is a short diagonal} If $\arc{a}$ is added to the interior of a tile of $\sT$, then in order to avoid $k$-neighbours, for $1 \leq k \leq m$, the only possibility is that $\arc{a}$ is added to a tile of type $(\mcT_{m+2})$ or $(\mcT^\prime_1)$. However, this would result in adjacent short diagonals, due to condition $(1)$ and the fact that these tiles have $2m$ isolated vertices. Hence, $\arc{a}$ must cross diagonals in $\sT$. In order to preserve the $m$-rigidity, $\arc{a}$ can only cross one set of long arcs incident with a common vertex $v_k$ of $\mcP_{n,m}$. Write $\arc{a}\coloneqq \{a_1, a_2\}$, such that $C(a_1,v_k,a_2)$. Due to Lemma \ref{lem:isos-shortdiagonals}, we must have at least $m$ isolated vertices preceding $a_1$ and $m$ isolated vertices following $a_2$. Moreover, since there are no $k$-neighbours, for $1 \leq k \leq m$, all $m-1$ vertices under the short arc $\arc{a}$ but $v_k$ are isolated in $\sT$. Moreover $a_1$ and $a_2$ cannot be incident with any other diagonal. Therefore, we can only add short arcs in the situation described in condition $(3)$. Since, by assumption this situation does not occur, the tiling gives rise to a maximal $m$-rigid object.  
\end{proof}

\begin{example}
In Figure \ref{fig:counter-example(3)}, we have a tiling of $\mcP_{8,2}$ which represents a $2$-rigid object but is not maximal. Indeed condition $(3)$ of Theorem \ref{thm:maximalmrigid} does not hold in this example, and one can add the dashed arc whilst preserving rigidity. 
\begin{figure}[!ht]
\begin{center}
\begin{tikzpicture}[thick,scale=0.6, every node/.style={scale=0.6}]

\draw (0,0) circle (3cm);

\draw (288:3cm) .. controls (302.4:2.6cm) .. (316.8:3cm);

\draw (28.8:3cm) .. controls (43.2:2.6cm) .. (57.6:3cm);

\draw (144:3cm) .. controls (158.4:2.6cm) .. (172.8:3cm);

\draw (28.8:3cm) .. controls (43.2:2cm) and (86.4:2cm) .. (100.8:3cm);

\draw (28.8:3cm) .. controls (57.6:1.2cm) and (115.2:1.2cm) .. (144:3cm);

\draw (28.8:3cm) .. controls (86.4:0.1cm) and (172.8:0.1cm) .. (230.4:3cm);

\draw (316.8:3cm) .. controls (331.2:2cm) and (14.4:2cm) .. (28.8:3cm);

\draw[dashed] (244.8:3cm) .. controls (230.4:2.6cm) .. (216:3cm);

\foreach \angle in {0,14.4,28.8,43.2,57.6,72,86.4,100.8,115.2,129.6,144,158.4,172.8,187.2,201.6,216,230.4,244.8,259.2,273.6,288,302.4,316.8,331.2,345.6}
{\filldraw[fill=black] (\angle:3cm) circle (0.07cm);}
\end{tikzpicture}
\end{center}
\caption{Tiling of $\mcP_{8,2}$ which does not satisfy condition $(3)$ in Theorem \ref{thm:maximalmrigid}.}
\label{fig:counter-example(3)}
\end{figure}
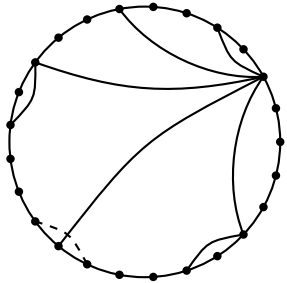
\end{example}

\begin{remark}
When $m=1$, condition (3) in Theorem \ref{thm:maximalmrigid} follows from condition (1). 
\end{remark}

%==============================================================================
%  Section
\section{Endomorphism algebras as tiling algebras}\label{sec:endoalgebras}
%==============================================================================

Our aim in this section is to study the endomorphism algebras of the connected maximal $m$-rigid objects characterised in the previous section. It turns out that these algebras are a subclass of a much larger class of algebras, which we shall call tiling algebras.

\begin{definition}
Let $\mcP$ be a disc with at least two marked points on the boundary. Given a tiling $\sT$ of $\mcP$, we associate a quiver $Q_{\sT}$ with relations $\sR_{\sT}$ to $\sT$ in the following manner:  

\textit{Vertices of $Q_{\sT}$}: The vertices correspond to all the interior diagonals of $\sT$.

\textit{Arrows of $Q_{\sT}$}: Two vertices of $Q_{\sT}$ are related by an edge in $Q_{\sT}$ if the corresponding diagonals of $\sT$ share a vertex and belong to the same tile. 

Orientation of edges: Let $\alpha, \beta$ be two diagonals of $\sT$ sharing a vertex $x$ of $\mcP$ and belonging to the same tile. The edge in $Q_{\sT}$ joining $\alpha$ and $\beta$ is oriented $\alpha \rightarrow \beta$ if the rotation with minimal angle around $x$ that sends $\alpha$ to $\beta$ is clockwise. 

\textit{Relations $\sR_{\sT}$ in $Q_{\sT}$}: The composition of two successive arrows coming from the same tile is zero. We denote by $\sI_{\sT}$ the ideal generated by the relations $\sR_{\sT}$. 

The algebra $A_{\sT} = (Q_{\sT}, \sI_{\sT})$ is called the \textit{tiling algebra associated to $\sT$}.
\end{definition}

\begin{example}
In Figure \ref{fig:tilingalgebra}, we have the tiling algebra corresponding to the tiling of $\mcP_{10,1}$ in Figure \ref{fig:tilesm=1}.
 
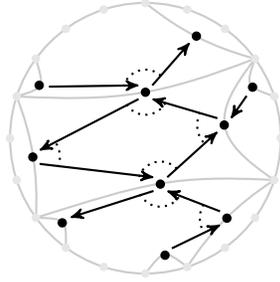
\begin{figure}[!ht]
\begin{center}
\begin{tikzpicture}[thick,scale=0.6, every node/.style={scale=0.6}, >=stealth']

\draw[black!20!white] (0,0) circle (3cm);

%tiles of type T_1

\draw[black!20!white] (18:3cm) .. controls (27.5:2.5cm) .. (36:3cm);

\draw[black!20!white] (144:3cm) .. controls (153.5:2.5cm) .. (162:3cm);

\draw[black!20!white] (216:3cm) .. controls (225.5:2.5cm) .. (234:3cm);

\draw[black!20!white] (270:3cm) .. controls (279.5:2.5cm) .. (288:3cm);

%tiles of type T'_1

\draw[black!20!white] (36:3cm) .. controls (54:2.5cm) and (72:2.5cm) .. (90:3cm);

\draw[black!20!white] (162:3cm) .. controls (180:2.5cm) and (198:2.5cm) .. (216:3cm);

\draw[black!20!white] (288:3cm) .. controls (306:2.5cm) and (324:2.5cm) .. (342:3cm);

\draw[black!20!white] (36:3cm) .. controls (72:1cm) and (126:1cm) .. (162:3cm);

\draw[black!20!white] (216:3cm) .. controls (270:1cm) and (306:1cm) .. (342:3cm);

\draw[black!20!white] (342:3cm) .. controls (0:1.5cm) and (18:1.5cm) .. (36:3cm); 

%%%quiver

\filldraw (25.5:2.63cm) circle (0.08cm); %a

\filldraw (153.5:2.63cm) circle (0.08cm); %i

\filldraw (225.5:2.63cm) circle (0.08cm); %j

\filldraw (279.5:2.63cm) circle (0.08cm); %g

\filldraw (63.5:2.53cm) circle (0.08cm); %h

\filldraw (189.5:2.53cm) circle (0.08cm); %e

\filldraw (315.5:2.53cm) circle (0.08cm); %f

\filldraw (90:1.02cm) circle (0.08cm); %d

\filldraw (288:1.07cm) circle (0.08cm); %c

\filldraw (9.5:1.77cm) circle (0.08cm); %b

\node (a) at (24.5:2.63cm) {};

\node (b) at (9.5:1.77cm) {};

\node (c) at (288:1.07cm) {};

\node (d) at (90:1.02cm) {};

\node (e) at (189.5:2.53cm) {};

\node (f) at (315.5:2.53cm) {};

\node (g) at (279.5:2.63cm) {};

\node (h) at (63.5:2.53cm) {};

\node (i) at (150.5:2.63cm) {};

\node (j) at (226.5:2.63cm) {};

\draw[->, name path=arrow10] (a.south west) -- (b.north east);

\draw[->, name path=arrow1] (b.north west) -- (d.south east);

\draw[->, name path=arrow2] (d.south west) -- (e.north east);

\draw[->, name path=arrow3] (e.south east) -- (c.north west);

\draw[->, name path=arrow4] (c.north east) -- (b.south west);

\draw[->, name path=arrow8] (f.north west) -- (c.south east);

\draw[->, name path=arrow9] (c.south west) -- (j.north east);

\draw[->, name path=arrow7] (g.north east) -- (f.south west);

\draw[->, name path=arrow6] (i.south east) -- (d.north west);

\draw[->, name path=arrow5] (d.north east) -- (h.south west);

\path[name path=Cc] (288:1.07cm) ellipse (0.7cm and 0.5cm);

\path[name intersections={of=arrow3 and Cc}];

\coordinate (3) at (intersection-1);

\path[name intersections={of=arrow4 and Cc}];
\coordinate (4) at (intersection-1);

\coordinate (0) at (288:1.07cm);

\tkzFindAngle(4,0,3) 
\tkzMarkAngle[dotted,size=0.5](4,0,3)

\path[name path=Cc] (288:1.07cm) ellipse (0.7cm and 0.5cm);

\path[name intersections={of=arrow9 and Cc}];

\coordinate (9) at (intersection-1);

\path[name intersections={of=arrow8 and Cc}];
\coordinate (8) at (intersection-1);

\coordinate (0) at (288:1.07cm);

\tkzFindAngle(9,0,8) 
\tkzMarkAngle[dotted,size=0.5](9,0,8)

\path[name path=Cc] (90:1.02cm) ellipse (0.7cm and 0.5cm);

\path[name intersections={of=arrow2 and Cc}];

\coordinate (2) at (intersection-1);

\path[name intersections={of=arrow1 and Cc}];
\coordinate (1) at (intersection-1);

\coordinate (0) at (90:1.02cm);

\tkzFindAngle(2,0,1) 
\tkzMarkAngle[dotted,size=0.5](2,0,1)

\path[name path=Cc] (90:1.02cm) ellipse (0.7cm and 0.5cm);

\path[name intersections={of=arrow5 and Cc}];

\coordinate (5) at (intersection-1);

\path[name intersections={of=arrow6 and Cc}];
\coordinate (6) at (intersection-1);

\coordinate (0) at (90:1.02cm);

\tkzFindAngle(5,0,6) 
\tkzMarkAngle[dotted,size=0.5](5,0,6)

\path[name path=Cc] (189.5:2.53cm) ellipse (0.5cm and 0.7cm);

\path[name intersections={of=arrow2 and Cc}];

\coordinate (2) at (intersection-1);

\path[name intersections={of=arrow3 and Cc}];
\coordinate (3) at (intersection-1);

\coordinate (0) at (189.5:2.53cm);

\tkzFindAngle(3,0,2) 
\tkzMarkAngle[dotted,size=0.6](3,0,2)

\path[name path=Cc] (9.5:1.77cm) ellipse (0.5cm and 0.7cm);

\path[name intersections={of=arrow4 and Cc}];

\coordinate (4) at (intersection-1);

\path[name intersections={of=arrow1 and Cc}];
\coordinate (1) at (intersection-1);

\coordinate (0) at (9.5:1.77cm);

\tkzFindAngle(1,0,4) 
\tkzMarkAngle[dotted,size=0.6](1,0,4)

\path[name path=Cc] (315.5:2.53cm) ellipse (0.5cm and 0.7cm);

\path[name intersections={of=arrow7 and Cc}];

\coordinate (7) at (intersection-1);

\path[name intersections={of=arrow8 and Cc}];
\coordinate (8) at (intersection-1);

\coordinate (0) at (315.5:2.53cm);

\tkzFindAngle(8,0,7) 
\tkzMarkAngle[dotted,size=0.6](8,0,7)

\foreach \angle in {0,18,36,54,72,90,108,126,144,162,180,198,216,234,252,270,288,306,324,342}
{\filldraw[fill=black!10!white,draw=black!10!white] (\angle:3cm) circle (0.07cm);}
\end{tikzpicture}
\end{center}
\caption{The tiling algebra of a tiling in $\mcP_{10,1}$.}
\label{fig:tilingalgebra}
\end{figure}
\end{example}

\subsection{Tiling algebras vs gentle algebras}
Before considering the endomorphism algebras of the maximal $m$-rigid objects studied in Section \ref{sec:classification}, we shall establish a relationship between tiling algebras and gentle algebras. 

From now on, paths in quivers will be read from left to right. Given a path $p$ in a quiver, we denote by $s(p)$ (resp. $t(p)$) the source (resp. target) of $p$. Given a vertex $x$ in a quiver, $v(x)$ denotes its valency.  
 
\begin{definition}
An algebra $A$ is {\it gentle} if it is Morita equivalent to an algebra $kQ/I$ satisfying the following conditions:
\begin{compactenum}
\item Each vertex of $Q$ is the starting point of at most two arrows and the endpoint of at most two arrows. 
\item For each arrow $\alpha$ in $Q$, there is at most one arrow $\beta$ in $Q$ such that $\alpha \beta \not\in I$, and there is at most one arrow $\gamma$ such that $\gamma \alpha \not\in I$.
\item For each arrow $\alpha$ in $Q$, there is at most one arrow $\delta $ in $Q$ such that $\alpha \delta \in I$, and there is at most one arrow $\mu$ such that $\mu \alpha \in I$.
\item $I$ is generated by paths of length 2. 
\end{compactenum}
\end{definition}

Throughout the remainder of the article, we assume that the quiver of a gentle algebra has no loops or 2-cycles.

Note that if a tiling has a tile with more than one open boundary, then the corresponding tiling algebra $A$ is disconnected. However, in this case $A$ is just a direct product of connected tiling algebras. Therefore, we restrict our attention to connected tiling algebras, and from now on every algebra will be tacitly considered to be connected, unless stated otherwise. 

\begin{definition}
Let $\kk Q/I$ be a path algebra. An oriented cycle in $Q$ is said to be {\it relation-full} if every pair of consecutive arrows in the cycle is a zero relation in $I$. 
\end{definition}

\begin{proposition}\label{prop:tilinggentle}
The tiling algebras are precisely the gentle algebras for which every cycle is oriented and relation-full. 
\end{proposition}
\begin{proof}
It is clear that tiling algebras are gentle algebras for which every cycle is oriented and relation-full. 

In the remainder of this subsection, we will introduce an algorithm, called the {\it tiling algorithm}, which constructs a tiling $\sT$ from a gentle algebra $\sG$ whose cycles are oriented and relation-full. It will be clear from the construction of the algorithm that the tiling algebra $A_{\sT}$ is $\sG$. 
\end{proof} 

Before describing the tiling algorithm, we need to recall some definitions (see \cite{AG}, for example) and set up some notation.  

\begin{definition} 
Let $\sG = \kk Q / I$ be a gentle algebra. 
\begin{compactenum}
\item A {\it (non-trivial) permitted path} in $\sG$ is a path $\alpha_1 \alpha_2 \ldots \alpha_k$ of length $k \geq 1$ with no relations. 
\item A maximal permitted path is called a {\it (non-trivial) permitted thread}. 
\item Let $v$ be a vertex of $Q$ which satisfies one of the following conditions:
\begin{compactitem}
\item The valency of $v$ is one, or 
\item The vertex $v$ is the source of exactly one arrow $\beta$, the target of exactly one arrow $\alpha$, and $\alpha \beta \not\in I$.
\end{compactitem} 
Then, we associate a {\it trivial permitted thread} to $v$, and denote it by $p_v$. 
\item A {\it (non-trivial) forbidden path} of $\sG$ is a path $\alpha_1 \alpha_2 \ldots \alpha_k$ of length $k \geq 1$ such that $\alpha_i \neq \alpha_j$ unless $i = j$, and $\alpha_i \alpha_{i+1} \in I$, for each $1 \leq i \leq k-1$. 
\item A maximal forbidden path is called a {\it (non-trivial) forbidden thread}. 
\item Let $v$ be a vertex of $Q$ which satisfies one of the following conditions:
\begin{compactitem}
\item The valency of $v$ is one, or 
\item The vertex $v$ is the source of exactly one arrow $\beta$, the target of exactly one arrow $\alpha$, and $\alpha \beta \in I$.
\end{compactitem} 
Then, we associate a {\it trivial forbidden thread} to $v$, and denote it by $f_v$. 
\item A forbidden thread $f$ is said to be {\it open} if it is not a cycle, and {\it closed} otherwise.
\end{compactenum}
\end{definition} 

Note that trivial forbidden threads are considered to be open. If every cycle in $\sG$ is oriented and relation-full, then every cycle in $\sG$ of length $r$ gives rise to $r$ forbidden threads. For each cycle $c$ in $\sG$, we choose one of these forbidden threads, and denote it by $f_c$. 

Let $\cF_{\sG} = \{\text{open forbidden threads in} \sG\} \cup \{f_c \mid c \text{ cycle in } \sG \}$. 

\begin{lemma}\label{lem:2forbidden}
Let $\sG$ be a gentle algebra whose cycles are oriented and relation-full, and $v$ be a vertex in the quiver of $\sG$. Then, there are precisely two forbidden threads in $\cF_{\sG}$ incident with $v$.
\end{lemma}
\begin{proof}
We have four cases, depending on the valency of $v$.

\Case{$v$ has valency 1} Let $\alpha$ be the arrow incident with $v$. Note that each arrow of the quiver of $\sG$ is used in precisely one forbidden thread in $\cF_{\sG}$. Hence, there is one and only one non-trivial forbidden thread in $\cF_{\sG}$ which uses $\alpha$. There is also one trivial forbidden thread associated to $v$. Clearly, these forbidden threads are distinct. 

\Case{$v$ has valency 2} If $v$ is a sink or a source, then there is no trivial forbidden thread at $v$, and there is a non-trivial forbidden thread for each arrow incident with $v$. Since a forbidden thread is an oriented path, these forbidden threads are distinct. 

Otherwise, let $\alpha, \beta$ be arrows in the quiver of $\sG$ such that $t(\alpha) = v = s(\beta)$. If $\alpha \beta$ is a relation, then there is a trivial forbidden thread at $v$ and a non-trivial forbidden thread using the arrows $\alpha$ and $\beta$. Finally, if there is no relation, then there are distinct forbidden threads using $\alpha$ and $\beta$, and no trivial forbidden thread at $v$.

\Case{$v$ has valency 3} Let $\alpha, \beta, \gamma$ be the arrows incident with $v$.  The vertex $v$ is either the source or the target of two of these arrows. Assume, without loss of generality, that $s(\alpha) = s(\beta) = t(\gamma) = v$. Moreover, assume $\gamma \alpha$ is a relation, and $\gamma \beta$ is not.  

We then have one non-trivial forbidden thread using $\gamma$ and $\alpha$ and one non-trivial forbidden thread using $\beta$. If they would coincide, then we would have a cycle which is not relation-full, contradicting the assumption.

\Case{$v$ has valency 4} Let $\alpha, \beta, \gamma, \delta$ be the arrows incident with $v$, and assume $t(\alpha) = t(\gamma) = s(\beta) = s(\delta) = v$. We must have precisely two relations, say $\alpha \beta$ and $\gamma \delta$. Arguing in the same manner as in the previous case, we have two distinct forbidden threads incident with $v$: one using $\alpha$ and $\beta$ and the other one using $\gamma$ and $\delta$.  
\end{proof}

We are now ready to describe the tiling algorithm, which is an iterative procedure. In each step, we consider a different forbidden thread in the set $\cF_{\sG}$ and construct the associated tile of the final tiling. 

{\it Initial step:} Choose a forbidden thread $f$ in $\cF_{\sG}$, and let $k$ be its length. 
If $f$ is open, draw $k+2$ marked points in a disc, label them by $1, 2, \ldots, k+2$, in the anticlockwise direction and draw diagonals linking $i$ and $i+1$, for $1 \leq i \leq k+1$. 

If $f$ is closed, draw $k$ marked points in a disc, label them by $1, 2, \ldots, k$, in the anticlockwise direction and draw diagonals linking $i$ and $i+1$, for $1 \leq i \leq k$ (where $k+1 = 1$). 

Finally, label the diagonal linking $i$ and $i+1$ by $d_i$.
 
From this initial step results a division of the disc into $x$ regions, $R_1, \ldots, R_{x}$, where $x = k+2$, if $f$ is open and $x = k+1$, if $f$ is closed. The region $R_1$ is bounded by all the diagonals $d_i$, and the only diagonal bounding $R_j$ is $d_{j-1}$, for each $2 \leq j \leq x$.

The region $R_1$ is one of the tiles of the tiling we are constructing, and we say that it is the tile, which we denote by $\mcR_f$, associated to $f$. Note that each diagonal drawn corresponds to a vertex of the path $f$.  

{\it Iterative step:} Let $\mcR_g$ be a tile constructed in a previous step, associated to a forbidden thread $g$, and bounded by a diagonal $d$. By Lemma \ref{lem:2forbidden}, the vertex of $\sG$ corresponding to $d$ is incident with precisely two forbidden threads. Let $g^\prime$ denote the other forbidden thread, and assume $g^\prime$ has not been considered in a previous step.

If $g^\prime$ is trivial, then we take the other region $R^\prime$ bounded by $d$ to be the tile $\mcR_{g^\prime}$ associated to $g^\prime$. 

Now suppose $g^\prime$ is not trivial and let $\ell$ be its length. If $g^\prime$ is open, add $\ell$ marked points in the open boundary of $R^\prime$. Otherwise, add $\ell-2$ marked points in the open boundary of $R^\prime$. 

Then add diagonals in the region $R^\prime$ linking these marked points and the endpoints of the diagonal $d$, in such a way that the quiver arising from the resulting tile $\mcR_{g^\prime}$ gives the forbidden thread $g^\prime$. Note that the diagonals must be added such that the full subquiver on the vertices of $g$ and $g^\prime$ agrees with the subquiver arising from $\mcR_g$ and $\mcR_{g^\prime}$. \\ 

Since $\cF_{\sG}$ is a finite set, this algorithm terminates. The output of this algorithm is a tiling of a marked disc. By construction, and since $\sG$ is connected, the tiling algebra corresponding to this tiling is $\sG$. This then finishes the proof of Proposition \ref{prop:tilinggentle}.  

\begin{remark}
In \cite[Subsection 3.1]{S14}, a procedure for constructing the Brauer graph associated to a gentle algebra is given and applied in the context of surface algebras, which were defined in \cite{DRS}. This construction is essentially a `dual version' of our tiling algorithm, in the sense that it uses permitted threads instead of forbidden threads. 

We expect that, using similar arguments to those used in \cite{S14} in the case of triangulations of a marked Riemann surface, one can prove that the Brauer graph algebra associated to a tiling of a marked disc (seen as a Brauer graph) is isomorphic to the trivial extension of the tiling algebra.
\end{remark} 

A \textit{fan} of a tiling of a marked disc $\mcP$ is the set of all diagonals incident with a marked point of $\mcP$. The following lemma, whose proof follows immediately from the construction of the tiling algebra, will be useful later.
 
\begin{lemma}\label{lem:fan-tiles}
Let $A_\sT$ be the tiling algebra of a tiling $\sT$ of a marked disc $\mcP$.
\begin{compactenum}
\item If there are at least two diagonals in $\sT$, then permitted threads in $A_{\sT}$ are in bijection with the non-isolated marked points of $\mcP$, or equivalently, with fans in $\sT$.
\item If $\sT$ has exactly one diagonal, then both non-isolated marked points of $\mcP$ correspond to the unique (trivial) permitted thread in $A_{\sT}$. 
\item A permitted thread in $A_{\sT}$ is trivial if and only if the corresponding marked point has valency one. 
\end{compactenum}
\end{lemma}

\subsection{Connected endomorphism algebras in $\sB_m (A_n)$}\label{subsec:endalg}

We will now turn our attention to the endomorphism algebras of connected maximal $m$-rigid objects in $\sB_m (A_n)$. Our aim is to give a complete description of these algebras in terms of quivers with relations.  

\begin{proposition}\label{prop:endtiling}
Let $T$ be a connected maximal $m$-rigid object of $\sB_m (A_n)$ and $\sT$ the corresponding tiling. The endomorphism algebra $\End_{\sB_m (A_n)} (T)$ is isomorphic to the tiling algebra $A_\sT$. 
\end{proposition}
\begin{proof}
The proof is similar to that of \cite[Theorem 5.2]{CS11}, taking into account \cite[Lemma 8.4]{CS11}, for the computation of Hom-spaces in $\sB_m (A_n)$. 
\end{proof}

\begin{remark}
We can apply the arguments in the proof of Proposition \ref{prop:endtiling} to any maximal $m$-rigid object $T$ to see that the corresponding endomorphism algebra is connected if and only if $T$ is connected.
\end{remark}

As a consequence of Propositions \ref{prop:tilinggentle} and  \ref{prop:endtiling}, the endomorphism algebras under consideration are gentle, with all cycles oriented and relation-full. The following result gives a characterisation of the endomorphism algebras of connected maximal $m$-rigid objects in $\sB_m (A_n)$ in terms of gentle algebras.

\begin{theorem}\label{thm:endalg}
Let $\sG = \kk Q/I$ be a gentle algebra whose cycles are oriented and relation-full. Then $\sG$ is the endomorphism algebra of a connected maximal $m$-rigid object in $\sB_m (A_n)$, for some $m, n \geq 1$, if and only if it satisfies the following conditions: 
\begin{compactenum}[(i)]
\item There is no permitted thread whose both source and target have valency one.
\item If $f$ is a forbidden path of length $m+1$ which cannot be completed to a closed forbidden thread, then $s(f)$ or $t(f)$ has valency one.
\item The length of every simple cycle is $m+3$.
\item Given $x \in Q_0$ of valency two such that $\alpha \beta \in I$, where $t(\alpha) = x = s(\beta)$, there are permitted threads $p_1, p_2$ with $t(p_1) = x = s(p_2)$ and $v(s(p_1)) = v(t(p_2)) = 1$.
\item If there is a forbidden thread $f$ of length $m+1$ with $v(t(f)) = 1$ (resp. $v(s(f)) = 1$), then there is a permitted thread $p$ such that $s(p) = s(f)$ and $v(t(p)) = 1$ (resp. $t(p) = t(f)$ and $v(s(p)) = 1$). 
\item There is no permitted thread $p$ with: $t(p) = t(f_1)$, $s(p) = s(f_2)$, for some open forbidden threads $f_1, f_2$ of lengths $\ell_1, \ell_2$, respectively, where $\ell_i \geq 2$ and $\ell_1 + \ell_2 \geq m+2$.  
\end{compactenum} 
\end{theorem}
\begin{proof}
Suppose $\sG$ is the endomorphism algebra of a connected maximal $m$-rigid object $T$ in $\sB_m (A_n)$, for some $n$. Then $\sG$ is the tiling algebra of a tiling with tiles as in Proposition \ref{prop:possibletiles} satisfying conditions $(1), (2)$ and $(3)$ in Theorem \ref{thm:maximalmrigid}. 

Since vertices of $Q$ of valency one correspond to short diagonals in the tiling, the algebra $\sG$ satisfies $(i)$ due to condition $(2)$. Condition $(iii)$ is satisfied, because cycles in $Q$ correspond to closed tiles in the tiling, and by Proposition \ref{prop:possibletiles}, closed tiles have length $m+3$. 

Now let $f$ be a forbidden path in $\sG$ of length $m+1$ which cannot be completed to a closed forbidden thread. Then $f$ is either a forbidden thread corresponding to a tile of type $(\mcT_{m+2})$ or there is an arrow $\alpha$ in $Q_1$ such that $f \alpha$ or $\alpha f$ is a forbidden thread corresponding to a tile of type $(\mcT_{m+3})$. Either way, the source or the target correspond to a short diagonal, meaning it has valency one. Hence $(ii)$ is satisfied. 

If a vertex has valency two with a relation, then it corresponds to the diagonal bounding a tile of type $(\mcT^\prime_1)$. Then condition $(1)$ implies $(iv)$.

Let $f$ be a forbidden thread of length $m+1$ such that $t(f)$ or $s(f)$ has valency one. Then $f$ corresponds to a tile of type $(\mcT_{m+2})$. The required permitted thread exists due to condition $(1)$. Hence, $(v)$ is also satisfied. 

Finally, suppose $(vi)$ does not hold. By Lemma \ref{lem:fan-tiles}, we have a fan at vertex $v$, say, and two tiles $\mcT_1, \mcT_2$, corresponding to $f_1$ and $f_2$, respectively, which satisfy the following: 
\begin{compactitem}
\item $\mcT_1, \mcT_2$ must be of type $(\mcT_{k_i})$, with $3 \leq k_i \leq m+2$, since the $f_i$ are open, $v(s(f_1)), v(t(f_2)) \geq 2$ and the length $\ell_i$ of $f_i$ is such that $k_i -1 = \ell_i \geq 2$, for $i = 1, 2$, 
\item $\mcT_1$ (resp. $\mcT_2$) have $x := m+k_1 -2$ (resp. $x^\prime := m+k_2 -2$) isolated vertices. In particular, $\mcT_1$ and $\mcT_2$ have at least $m+1$ isolated vertices in their open boundary.
\item $x + x^\prime +1 \geq 3m +1$, since $\ell_1 +\ell_2 \geq m+2$. 
\end{compactitem}

Hence, (3) does not hold, a contradiction. Therefore, $\sG$ satisfies conditions $(i), \ldots, (vi)$. 

Conversely, let $\sG$ be a gentle algebra such that cycles are oriented and relation-full and conditions $(i), \ldots, (vi)$ are satisfied for some $m \geq 1$. By Proposition~\ref{prop:tilinggentle}, $\sG$ is a tiling algebra. Let $\sT$ be a tiling of a marked disc corresponding to $\sG$ via the tiling algorithm. Recall that open forbidden threads correspond to open tiles. 

{\it Claim 1:} If $f$ is an open forbidden thread of length $m+2$, then $v(s(f)) = v(t(f)) = 1$

Indeed, given an open forbidden thread $f = \alpha_1 \cdots \alpha_{m+2}$ of length $m+2$, we have that $v(s(\alpha_1)) = 1$ or $v(t(\alpha_{m+1})) = v(s(\alpha_{m+2})) = 1$, by condition $(ii)$. But $v(t(\alpha_{m+1})) \geq 2$, and so $v(s(f)) = v(s(\alpha_1)) = 1$. 

Similarly, by condition $(ii)$ we have $v(s(\alpha_2)) = 1$ or $v(t(\alpha_{m+2})) = 1$. However, $v(s(\alpha_2)) \geq 2$, and so $v(t(f)) = v(t(\alpha_{m+2})) = 1$, which proves the claim. 

{\it Claim 2:} The maximum length of an open forbidden thread in $\sG$ is $m+2$. 

Indeed, suppose $f = \alpha_1 \cdots \alpha_\ell$ is an open forbidden thread of length $\ell \geq m+3$. By condition $(ii)$, we have that the valency of $s(\alpha_2)$ or $t(\alpha_{m+2})$ must be one, a contradiction. 

By claim 2, each open tile of $\sT$ has at most length $m+3$. Let $\mcT$ be an open tile with length $k$, where $1 \leq k \leq m+3$. 

If $2 \leq k \leq m+2$ (resp. $k = m+3$), add $m+k-2$ (resp. $m$) isolated vertices in the open boundary of $\mcT$. For $k = 1$, add $m-1$ (resp. $2m$) isolated vertices in the open boundary of $\mcT$ if the vertex in $Q$ corresponding to the interior arc bounding $\mcT$ has valency one (resp. two). 

Due to claim 1 and conditions $(ii)$ and $(iii)$, we can easily deduce that the tiling $\sT$ is formed by tiles of types $(\mcT_k)$, with $1 \leq k \leq m+3$, $(\mcT^\prime_1)$ or $(\mcT_{m+3}^\prime)$. 

By condition $(i)$, there are no adjacent short diagonals in $\sT$. Hence (2) is satisfied. 

The only cases where there are $2m$ consecutive isolated vertices are in tiles of type $(\mcT^\prime_1)$ or $(\mcT_{m+2})$. In tiles of type $(\mcT^\prime_1)$, the bounding interior arc corresponds to a vertex of valency $2$, and the arrows incident with this vertex arise from the other tile incident with the interior arc, and so there is a relation. By $(iv)$, we have the short arcs required in condition (1). In tiles of type $(\mcT_{m+2})$, one of the short arcs is already a bounding interior arc of the tile, and the existence of the other short arc is guaranteed by condition $(v)$. Hence $\sT$ satisfies (1). 

Finally, suppose $\sT$ does not satisfy (3). As we have seen above, this would imply that $\sG$ does not satisfy $(vi)$, a contradiction.    
\end{proof}

%==============================================================================
% Section
\section{Applications}

In the final section of this article, we demonstrate some ways in which the nice combinatorial presentation of the endomorphism algebras of connected maximal $m$-rigid objects in $\sB_m (A_n)$ can be used to get a better understanding of these algebras. In the first subsection, we will relate them to (higher) cluster-tilted algebras of type $A$, and in the remaining two subsections, we will study two homological properties, namely Gorenstein dimensions and AG-invariance. We in fact study these latter two properties in the more general framework of tiling algebras, before then applying the results to our endomorphism algebras.

\subsection{Relationship with $(m+1)$-cluster-tilted algebras}\label{sec:clustertilted}

Using the geometric model of $(m+1)$-cluster categories of type $A_n$ from \cite{BM}, we can view $(m+1)$-cluster-tilted algebras of type $A_n$ as tiling algebras of a disc with $(n+1)(m+1)+2$ marked points, whose tiles are $(m+3)$-gons (see \cite{Murphy}). In order to compare our endomorphism algebras with these algebras, we need to recall the notion of a cut of a quiver. 

\begin{definition}\cite{Fernandez}
Let $Q$ be a quiver and $C$ the set of all full subquivers of $Q$ given by simple cycles. Any subset of the set of arrows lying in $C$ is called a {\it cut} of $Q$.

Let $A = \kk Q / I$ be a quotient of the path algebra of $Q$ by an admissible ideal $I$. An algebra is said to be obtained from $A$ by a cut if it is isomorphic to a quotient $\kk Q /\langle I \cup C \rangle$, where $C$ is a cut of $Q$.  
\end{definition} 

\begin{proposition}
Let $T$ be a connected maximal $m$-rigid object in $\sB_m (A_n)$ given by the tiling $\sT$ of $\mcP_{n,m}$, and let $A_{\sT}$ be the corresponding endomorphism algebra. Denote by $n_k$ the number of tiles in $\sT$ of type $(\mcT_k)$, for each $1 \leq k \leq m+3$, and by $n^\prime_1$ the number of tiles in $\sT$ of type $(\mcT_1^\prime)$. 
\begin{compactenum}
\item If $n_{m+3} = 0$, then $A_{\sT}$ is an $(m+1)$-cluster-tilted algebra of type $A_{n^\prime}$, with $n^\prime = n+\frac{x-4}{m+1}$, where $x \coloneqq (1-m) n_1^\prime - \sum\limits_{k=1}^{m+2} (2k-4) n_k$. 
\item If $n_{m+3} \neq 0$, then there is an $(m+1)$-cluster-tilted algebra of type $A_{n^{\prime \prime}}$, with $n^{\prime \prime} = n-n_{m+3}+\frac{x-4}{m+1}$, and $x$ as above, from which $A_{\sT}$ can be obtained via a cut.
\end{compactenum}
\end{proposition}
\begin{proof}
By Theorem \ref{thm:maximalmrigid}, $\sT$ is made of tiles of type $(\mcT_k)$, with $1 \leq k \leq m+3$, $(\mcT_1^\prime)$ or $(\mcT^\prime_{m+3})$. By changing the number of isolated vertices, we can make each tile (except those of type $(\mcT_{m+3})$) into an $(m+3)$-gon. Indeed, each tile of type $(\mcT_k)$, with $1 \leq k \leq m+2$ is a $(2k+m-1)$-gon with $m+k-2$ isolated vertices. Hence, if we remove $2k-4$ of these isolated vertices (which means adding two vertices for $k=1$), the tile becomes an $(m+3)$-gon. Tiles of type $(\mcT^\prime_{m+3})$ are already $(m+3)$-gons, so no changes need to be made to these. Finally, tiles of type $(\mcT_1^\prime)$  are $(2m+2)$-gons, and so by removing $m-1$ of their $2m$ isolated vertices, we obtain an $(m+3)$-gon. 

If $\sT$ has no tiles of type $(\mcT_{m+3})$, then by altering the number of isolated vertices as explained above, we get an $(m+3)$-angulation of an $N$-gon, where $N \coloneqq (n+1) (m+1)-2 - (m-1)n_1^\prime - \sum\limits_{k=1}^{m+2} (2k-4) n_k$. Note that $N$ must be 2 modulo $m+1$ (cf. \cite[Lemma 2.7]{Murphy}). Indeed, it can be written in the form $(m+1) (n^\prime +1) +2$, where $n^\prime$ is as in the statement of the theorem. 

Tiles of type $(\mcT_{m+3})$ are $(2m+4)$-gons, and so we would have to remove $m+1$ isolated vertices from such a tile in order to convert it into an $(m+3)$-gon. However, this is impossible since such a tile only has $m$ isolated vertices. But it is easily seen that this type of tile can be obtained from a closed $(m+3)$-gon by performing a cut, followed by adding $m$ isolated vertices in the open boundary.  
\end{proof}

\subsection{Gorenstein dimension}\label{sec:Gorenstein}

A gentle algebra $\sG$ has finite injective dimension as both a left and a right $\sG$-module (cf. \cite{GR}). This dimension is called the {\it Gorenstein dimension} of $\sG$, and we denote it by $Gdim (\sG)$. This concept took its inspiration from commutative ring theory, and has been a subject of interest in cluster-tilting theory.

By applying a result of \cite{GR}, we are now able to calculate the Gorenstein dimensions of tiling algebras and, in particular, of the endomorphism algebras we consider. We start by stating the result of \cite{GR} that we will apply, noting that a {\it gentle arrow} $\alpha$ in $\sG$ is an arrow for which there is no $\alpha_0$ such that $\alpha_0 \alpha$ is a relation. 

\begin{theorem}\cite{GR}\label{thm:GR}
Let $\sG$ be a gentle algebra, and let $n (\sG)$ be the maximum length of a forbidden path starting with a gentle arrow, or zero if there are no gentle arrows. 
\begin{compactenum}
\item $n(\sG)$ is smaller or equal to the number of arrows in $\sG$.
\item If $n(\sG) > 0$, then $\sG$ has Gorenstein dimension $n(G)$.
\item If $n(\sG) = 0$, then $\sG$ has Gorenstein dimension at most one. 
\end{compactenum}
\end{theorem}

It was shown in \cite[Theorem 2.7]{ABCP} that cluster-tilted algebras of type $A$ have Gorenstein dimension one. The following gives a further result in this direction, and can be considered as a partial generalisation.

\begin{lemma}\label{lem:Gorenstein}
Let $\sT$ be a tiling of a marked disc with an open tile of length at least two. Let $k$ be the maximum length of an open tile of $\sT$, and let $A_{\sT}$ be the tiling algebra associated to $\sT$. Then, $Gdim (A_{\sT}) = k-1$. 
\end{lemma}
\begin{proof}
Because there is an open tile with at least two interior bounding arcs, $\sT$ has a gentle arrow. In fact, the forbidden threads of $\sT$ starting with a gentle arrow are precisely those arising from such tiles. Since the length of each such forbidden thread is given by subtracting one from the length of the corresponding tile, we have $n (A_{\sT}) = k-1 \geq 1$, and the result follows immediately by Theorem~\ref{thm:GR}.
\end{proof}

We note that it is immediate from Theorem~\ref{thm:GR} that those tiling algebras not covered by Lemma~\ref{lem:Gorenstein} have Gorenstein dimension at most one.

\begin{corollary}
Let $T$ be a connected maximal rigid object in $\sB_m (A_n)$, $A_T$ its endomorphism algebra and $\sT$ the corresponding tiling.

If $n = 2$, then $Gdim (A_T) = 0$. If $n \geq 3$, then $Gdim (A_T) = \mathsf{max} \{k \mid k \geq 1 \text{ and } (\mcT_{k+1}) \text{ is a tile of } \sT\}$. In particular, $1 \leq Gdim (A_T) \leq m+2$. 
\end{corollary}
\begin{proof}
Let $n = 2$. Then all summands of $T$ correspond to short arcs, and since there cannot be adjacent short arcs nor crossings, $T$ has only one summand, and so $A_T$ is of type $A_1$. In this case, $A_T$ is self-injective, and so its Gorenstein dimension is zero.

Now suppose $n \geq 3$. Note that no connected maximal $m$-rigid objects correspond to tilings consisting simply of tiles of types $(\mcT_1), (\mcT_1^\prime)$ and $(\mcT_{m+3}^\prime)$. The result then follows from Lemma \ref{lem:Gorenstein} and from the characterisation of the possible tiles. 
\end{proof}

\subsection{The AG-invariant of tiling algebras}

The AG-invariant, introduced in \cite{AG}, is a derived invariant for gentle algebras, and is computed combinatorially via the quiver and relations of the algebra. It consists of a set of ordered pairs generated by the algorithm which follows this preliminary definition. 

\begin{definition}
Let $H$ be a permitted thread. We define the {\it forbidden thread $F$ ending at $t(H)$ from the opposite direction to $H$} as follows:

\Case{1} $t(H)$ is the target of two arrows, $\alpha$ and $\beta$.

\noindent If $\alpha$ is the final arrow of $H$, then $F$ is the forbidden thread whose final arrow is $\beta$.

\Case{2} $t(H)$ is the target of at most one arrow $\alpha$.

\noindent If $H$ does not contain $\alpha$, then $F$ is the forbidden thread whose final arrow is $\alpha$. Otherwise, $F = f_{t(H)}$ (the trivial forbidden thread on $t(H)$). \vspace*{6pt}

Given a forbidden thread $F$, we define the {\it permitted thread $H$ starting at $s(F)$ in the opposite direction to $F$} in a similar manner. 
\end{definition}

We now state the AG-invariant algorithm:
\begin{compactenum}
\item Start with a permitted thread $H_0$. 
\item To each permitted thread $H_i$, let $F_i$ be the forbidden thread ending at $t(H_i)$ from the opposite direction to $H_i$. 
\item To each forbidden thread $F_i$, let $H_{i+1}$ be the permitted thread starting at $s(F_i)$ in the opposite direction to $F_i$. 
\item Continue in this manner until the first permitted thread appears again, i.e. $H_{a} = H_0$. This gives rise to a pair $(a,b)$, where $b$ is the number of arrows used in the forbidden threads $F_i$.
\item Repeat this procedure until every permitted thread has appeared.
\item For each oriented and relation-full cycle, associate a pair $(0,c)$, where $c$ is the length of the cycle.  
\end{compactenum}
We note that the AG-invariant of a gentle algebra is independent of the choices of (initial) permitted threads made in the algorithm.

\begin{remark}\label{rmk:AG}\cite[Remark 8]{AG}
Let $\sG$ be a gentle algebra with AG-invariant $AG (\sG)\coloneqq \{(a_1,b_1), \ldots, (a_k,b_k)\}$. Then $\sum\limits_{j=1}^k a_j$ is the number of permitted threads, and coincides with the number of open forbidden threads. Moreover, $\sum\limits_{j=1}^k b_j$ is the number of arrows of the quiver of $G$. 
\end{remark} 

Recall that permitted threads of a tiling algebra correspond to fans in the tiling (see Lemma \ref{lem:fan-tiles}). Given a fan at a marked point $v$, we shall denote the corresponding permitted thread by $p_v$. Recall also that forbidden threads come from tiles. 

We will now describe the AG-invariant of a tiling algebra in terms of the tiling. Similar ideas were used in \cite{DR,DRS}, where $(m+2)$-angulations and surface algebras (special cases of tiling algebras when the surface is a disc) were considered. But we include the proof for the convenience of the reader.  

\begin{lemma}\label{lem:oneorbit}
Let $A_{\sT}$ be the tiling algebra of a tiling $\sT$. Then, all the permitted threads of $A_{\sT}$ appear in a single orbit arising from the AG-invariant algorithm. In particular, there is precisely one ordered pair $(a,b)$ for which $a>0$. 
\end{lemma}
\begin{proof}
We shall prove this lemma by induction on the number of tiles of $\sT$. If $\sT$ has two tiles, then $A_{\sT}$ is of Dynkin type $A_1$, and the lemma is obvious.

Assume that the lemma holds for any tiling algebra $A_{\sT^\prime}$ whose tiling $\sT^\prime$ has $r$ tiles.

Now, suppose $\sT$ has $r+1$ tiles, and let $\alpha_1$ be an arrow of $A_\sT$. We know that there is exactly one permitted thread $p_1$ and exactly one forbidden thread $f_1 \in \cF_{A_\sT}$ passing through $\alpha_1$. Suppose, without loss of generality, that $s(\alpha_1) = s(f_1)$ and write $f_1 = \alpha_1 \alpha_2\ldots \alpha_{k}$. Let $\mcR$ be the corresponding tile. Let $\{v_i,v_{i+1}\}$ be the diagonal corresponding to $s(\alpha_i)$, for each $i = 1, \ldots, k$ (noting of course that $v_{k+1} = v_1$ if $\mcR$ is a closed tile). Denote the region bounded by the diagonal $\{v_i, v_{i+1}\}$, but that does not contain $\mcR$, by $\sR_i$. Note that we can see each region $\sR_i$, together with the diagonal $\{v_i,v_{i+1}\}$, as a tiling of a marked disc with fewer marked points. We can then modify this tiling by adding a new boundary arc, joining $v_i$ and $v_{i+1}$, such that the original diagonal $\{v_i, v_{i+1}\}$ becomes an interior arc. Denote this modified tiling by $\overline{\sR}_i$, and the corresponding tiling algebra by $A_{\overline{\sR}_i}$

We can then write $p_1 = p \alpha_1 p^\prime$, where $p$ is the permitted thread of $A_{\overline{\sR}_1}$ corresponding to the fan at $v_{2}$ in $\overline{\sR}_1$, and $p^\prime$ is the permitted thread of $A_{\overline{\sR}_{2}}$ corresponding to the fan of $v_{2}$ in $\overline{\sR}_{2}$.  

By the induction hypothesis, we get a single orbit with all the permitted threads of $A_{\overline{\sR}_{2}}$, which can be written as:
$H^2_0 = p^\prime, \ldots, H^2_{\ell_2} = p_{v_{3}} \text{(in $A_{\overline{\sR}_2}$)}, H^2_{\ell_2+1} = H^2_0$, and where $F^2_{\ell_2} = f_{s(p^\prime)}$. (We observe that the trivial forbidden thread $f_{s(p^\prime)}$ exists in $A_{\overline{\sR}_2}$, but not in $A_\sT$.)

Now, in $A_\sT$, the orbit starting at $H_0 = p_1$ is such that $F_0 = F_0^2$ and $H_j= H^2_j, F_j = F^2_j$, for $1 \leq j \leq \ell_2-1$ (and hence contains all of the permitted threads of $A_{\overline{\sR}_2}$ that arise in $\sT$). In addition, we have $H_{\ell_2} = H^2_{\ell_2} \alpha_2 H^3_0$, where $H^3_0$ is the permitted thread in $A_{\overline{\sR}_3}$ corresponding to the fan at $v_3$. With successive applications of the induction hypothesis for $A_{\overline{\sR}_3}, \ldots , A_{\overline{\sR}_{k-1}}$, it follows that we can continue this orbit up to and including $H_{\ell_2 + \ldots + \ell_{k-1}} = H^{k-1}_{\ell_{k-1}} \alpha_{k-1} H^k_0$ (by which point all of the permitted threads from $A_{\overline{\sR}_2}, \ldots , A_{\overline{\sR}_{k-1}}$ which appear in $\sT$ have already arisen), where $H^k_0$ is the permitted thread in $A_{\overline{\sR}_k}$ corresponding to the fan at $v_k$.
 
By a further application of the induction hypothesis, we have that all of the permitted threads $A_{\overline{\sR}_{k}}$ occur in a single orbit, which can be written as $H^k_0, \ldots, H^k_{\ell_k} = p_{v_{k+1}} \text{(in $A_{\overline{\sR}_k}$)}, H^k_{\ell_k+1} = H^k_0$, and where $F^k_{\ell_k} = f_{s(H^k_0)}$.

Now, if $\mcR$ is open, then the fan at $v_{k+1}$ in $\overline{\sR}_k$ is the same as the fan at $v_{k+1}$ in $\sT$, and so $H^k_{l_k}$ is a permitted thread of $A_T$. In fact, $H_{\ell_2 + \ldots + \ell_k} = H^k_{\ell_k}$. We then have $F_{\ell_2 + \ldots + \ell_k} = f_1$, and that $H_{\ell_2 + \ldots + \ell_k +1}$ is the permitted thread of $v_1$ in $\sT$, which coincides with the permitted thread of $v_1$ in $\overline{\sR}_1$. With a final application of the induction hypothesis, we then get all the permitted threads of $A_\sT$ lying in $\overline{\sR}_1$, finishing with the permitted thread of $v_2$ in $\sT$, which is again our initial permitted thread $H_0$. This completes our orbit. The orbit we have obtained contains all of the permitted threads in $\sT$, as required, as it contains all of the permitted threads in $\sT$ that appear in $A_{\overline{\sR}_1}, \ldots , A_{\overline{\sR}_{k}}$, as well as each permitted thread in $\sT$ containing $\alpha_i$ for $i = 1, \ldots, k$.

If $\mcR$ is closed, then $H_{\ell_2 + \ldots + \ell_k} = H^k_{\ell_k} \alpha_k p_{v_1}$, where $p_{v_1}$ is the permitted thread corresponding to $v_1$ in $\overline{\sR}_1$. Again, with a final application of the induction hypothesis, we then get all the permitted threads of $A_\sT$ lying in $\overline{\sR}_1$, finishing with the permitted thread of $v_2$ in $\sT$. We have thereby again obtained a single orbit which contains all of the permitted threads of $\sT$.
\end{proof}

The proofs of the following corollaries follow immediately from Lemma \ref{lem:oneorbit} and Remark \ref{rmk:AG}.

\begin{corollary}
Let $A_\sT$ be the tiling algebra of a tiling $\sT$. The AG-invariant of $A_\sT$ is as follows:
\begin{compactenum}
\item There is an ordered pair $(0,c)$ in AG($A_\sT$) if and only if there is a closed tile in $\sT$ with length $c$.
\item There is only one ordered pair $(a,b)$ in AG($A_\sT$) with $a \neq 0$. Moreover, $a$ is the number of open tiles in $\sT$, and
$$b = \sum\limits_{\mcT \text{ open tile in $\sT$}} (\ell (\mcT) - 1) = -a + \sum\limits_{\mcT \text{ open tile in $\sT$}} \ell (\mcT).$$
\end{compactenum}
\end{corollary}

\begin{corollary}
Let $A$ be the endomorphism algebra of a connected maximal $m$-rigid object in $\sB_m (A_n)$, and $\sT$ the corresponding tiling. Let $n_k, n^\prime_1, n^\prime_{m+3}$ be the number of tiles in $\sT$ of types $(\mcT_k), (\mcT_1^\prime), (\mcT^\prime_{m+3})$, respectively, where $1 \leq k \leq m+3$. 

Then the AG-invariant of $A$ is given by a sequence of $n^\prime_{m+3}$ ordered pairs of the form $(0,m+3)$, together with an ordered pair $(a,b)$, with $a= \sum\limits_{k=1}^{m+3} n_k + n^\prime_1$, and $b = \sum\limits_{k=1}^{m+3} (k-1)n_k$.
\end{corollary}

In \cite{DR}, the description of the AG-invariant of the algebras considered was via marked points and the lengths of the open boundaries. In our set-up, we cannot describe the AG-invariant in terms of lengths of open boundaries, because different tiles, which give rise to different numbers of arrows, can have open boundaries with the same length.

\subsection*{Acknowledgements}
The first author would like to thank Funda\c{c}\~ao para a Ci\^encia e Tecnologia, for their financial support through Grant SFRH/BPD/90538/2012. The second author gratefully acknowledges support by the Austrian Science Fund (FWF): Project No.\ P25141-N26, and NAWI Graz. The authors also respectively thank the University of Graz and the University of Lisbon for their kind hospitality during research visits.

%==============================================================================
%  Bibliography

%==============================================================================

\end{document}